\documentclass[a4paper,12pt]{article}

\usepackage{amsfonts}
\usepackage{amscd,color}
\usepackage{amsmath,amsfonts,amssymb,amscd}
\usepackage{indentfirst,graphicx,epsfig}
\input{epsf}
\usepackage{graphicx}
\usepackage{epstopdf}
\usepackage{caption}
\usepackage{subfigure}
\usepackage{mathrsfs}
\usepackage{ifpdf}

\newtheorem{thm}{Theorem}[section]
\newtheorem{df}[thm]{Definition}

\newtheorem{lem}[thm]{Lemma}
\newtheorem{ob}[thm]{Observation}

\newtheorem{cor}[thm]{Corollary}
\newtheorem{pro}[thm]{Proposition}
\newenvironment {proof} {\noindent{\em Proof.}}{\hspace*{\fill}$\Box$\par\vspace{4mm}}
\newcommand{\ml}{l\kern-0.55mm\char39\kern-0.3mm}

\title{\textbf{ Strong conflict-free connection of
graphs\footnote{Supported by NSFC No.11871034, 11531011 and NSFQH No.2017-ZJ-790.}}}
\author{{\small Meng Ji$^1$, Xueliang Li$^{1,2}$} \\
{\small  $^1$Center for Combinatorics and LPMC}\\
{\small Nankai University, Tianjin 300071, China}\\
{\small jimengecho@163.com,}
{\small lxl@nankai.edu.cn }\\
\small $^2$School of Mathematics and Statistics, Qinghai Normal University\\
\small Xining, Qinghai 810008, China\\
}
\date{}
\begin{document}
\maketitle
\begin{abstract}
A path $P$ in an edge-colored graph is called \emph{a conflict-free path} if there exists a color used on only one of the edges of $P$. An edge-colored graph $G$ is called \emph{conflict-free connected} if for each pair of distinct vertices of $G$ there is a conflict-free path in $G$ connecting them. The graph $G$ is called \emph{strongly conflict-free connected }if for every pair of vertices $u$ and $v$ of $G$ there exists a conflict-free path of length $d_G(u,v)$ in $G$ connecting them. For a connected graph $G$, the \emph{strong conflict-free connection number} of $G$, denoted by $\mathit{scfc}(G)$, is defined as the smallest number of colors that are required in order to make $G$ strongly conflict-free connected. In this paper, we first show that if $G_t$ is a connected graph with $m$ $(m\geq 2)$ edges and $t$ edge-disjoint triangles, then $\mathit{scfc}(G_t)\leq m-2t$, and the equality holds if and only if $G_t\cong S_{m,t}$. Then we characterize the graphs $G$ with $scfc(G)=k$ for $k\in \{1,m-3,m-2,m-1,m\}$. In the end, we present a complete characterization for the cubic graphs $G$ with $scfc(G)=2$.
\\[2mm]
\textbf{Keywords:} strong conflict-free connection coloring (number); characterization; cubic graph\\
\textbf{AMS subject classification 2010:} 05C15, 05C40, 05C75.\\
\end{abstract}

\section{Introduction}

All graphs mentioned in this paper are simple, undirected and finite. We follow book \cite{BM} for undefined notation and terminology. Coloring problems are important topics in graph theory. In recent years, there have appeared a number of colorings raising great concern due to their wide applications in real world. We list a few well-known colorings here. The first of such would be the rainbow connection coloring, which is stated as follows. A path in an edge-colored graph is called a {\it rainbow path} if all the edges of the path have distinct colors. An edge-colored graph is called ($strongly$) rainbow connected if there is a ($shortest$ and) rainbow path between every pair of distinct vertices in the graph. For a connected graph $G$, the ($strong$) rainbow connection number of $G$ is defined as the smallest number of colors needed to make $G$ ($strongly$) rainbow connected, denoted by ($\mathit{src}(G)$) $\mathit{rc}(G)$. These concepts were first introduced by Chartrand et al. in \cite{CJMZ}.

Inspired by the rainbow connection coloring, the concept of proper connection coloring was independently posed by Andrews et al. in \cite{ALLZ} and Borozan et al. in \cite{BFGMMMT}, the only difference from ($strong$) rainbow connection coloring is that distinct colors are only required for adjacent edges instead of all edges on the ($shortest$) path. For an edge-colored connected graph $G$, the smallest number of colors required to give $G$ a ($strong$) proper connection coloring is called the \emph{(strong) proper connection number} of $G$, denoted by  $(\mathit{spc}(G))$ $\mathit{pc}(G)$.

The hypergraph version of conflict-free coloring was first introduced
by Even et al. in \cite{ELR}. A hypergraph $H$ is a pair $H=(X,E)$ where $X$ is the set of vertices, and $E$ is the set of nonempty subsets of $X$, called hyperedges. The coloring was motivated to solve the problem of assigning frequencies to different base stations in cellular networks, which is defined as a vertex-coloring of $H$ such that every hyperedge contains a vertex with a unique color.

Later on, Czap et al. in \cite{CJV} introduced the concept of \emph{conflict-free connection coloring} of graphs, motivated by the earlier hypergraph version. A path in an edge-colored graph $G$ is called a \emph{conflict-free path} if there is a color appearing only once on the path. The graph $G$ is called \emph{conflict-free connected} if there is a conflict-free path between each pair of distinct vertices of $G$. For a connected graph $G$, the minimum number of colors required to make $G$ conflict-free connected is defined as the \emph{conflict-free connection number} of $G$, denoted by $\mathit{cfc}(G)$. For more results, the reader can be referred to \cite{CHLMZ,CJLZ,CJMZ,LZZMZJ}.

In this paper, we focus on studying the strong conflict-free connection coloring which was introduced by Ji et al. in \cite{JLZ}, where only computational complexity was studied. An edge-colored graph is called \emph{strongly conflict-free connected }if there exists a conflict-free path of length $d_G(u,v)$ for every pair of vertices $u$ and $v$ of $G$. For a connected graph $G$, the \emph{strong conflict-free connection number} of $G$, denoted $\mathit{scfc}(G)$, is the smallest number of colors that are required to make $G$ strongly conflict-free connected.

The paper is organized as follows. In Section \ref{preliminaries}, we give some preliminary results. In Section \ref{The bound}, we show that if $G_t$ is a connected graph with $m$ $(m\geq 2)$ edges and $t$ edge-disjoint triangles, then $\mathit{scfc}(G_t)\leq m-2t$, and the equality holds if and only if $G_t\cong S_{m,t}$. In Section \ref{Large and small}, we characterize the graphs $G$ with $scfc(G)=k$ for $k\in\{1,m-3,m-2,m-1,m\}$. In the last section, we completely characterize the cubic graphs $G$ with $scfc(G)=2$.

\section{Preliminaries}\label{preliminaries}

In this section, we present some results which will be used in the sequel. In \cite{JLZ}, the authors obtained the following computational complexity result.

\begin{thm}{\upshape\cite{JLZ}}
For a connected graph $G$ and integer $k\geq 2$, deciding whether $\mathit{scfc}(G)\leq k$ is NP-complete.
\end{thm}

They also showed the following result.
\begin{thm}\upshape\label{rc-cfc}\cite{JLZ}
For a graph $G$, $\mathit{rc}(G) = 2$ if and only if $diam(G) = 2$  and $\mathit{scfc}(G) = 2$.
\end{thm}

Note that from \cite{CJMZ}, one has that $rc(G) = 2$ if and only if $src(G) = 2$.
The following result is obvious.
\begin{thm}\label{scfc tree}
For a tree $T$, $\mathit{scfc}(T)=\mathit{cfc}(T)$. Therefore, for a path $P_n$ on $n$
vertices, $\mathit{scfc}(P_n)=\lceil\log_2n\rceil$; for a star $S_m$ with $m$ edges, $\mathit{scfc}(S_m)=m$.
\end{thm}

The authors in \upshape \cite{CJMZ} obtained the strong rainbow connection number for a wheel graph $W_n$, where $n$ is the degree of the central vertex.
\begin{thm}{\upshape \cite{CJMZ}}\label{rainbowconnection}
For $n\geq3$, let $W_n$ be a wheel. Then $\mathit{src}(W_n)=\lceil\frac{n}{3}\rceil$.
\end{thm}

For a complete bipartite graph $K_{s,t}$,
they also got the following result.

\begin{thm}\upshape\cite{CJMZ}\label{K-s-t}
For integers $s$ and $t$ with $1\leq s\leq t$, $\mathit{src}(K_{s,t})=\lceil\sqrt[s]{t}\rceil$.
\end{thm}

From the above results, we get that

\begin{thm}
$\mathit{scfc}(W_n)=\lceil\frac{n}{3}\rceil$.
\end{thm}
\begin{proof} Note that for a graph $G$ with diameter 2, a strong rainbow path (of length 2) of $G$ is a strong conflict-free path of $G$, and vice versa.
Since $diam(W_n)=2$, then $\mathit{scfc}(W_n)=\mathit{src}(W_n)$. So, $\mathit{scfc}(W_n)=\lceil\frac{n}{3}\rceil$ from Theorem \ref{rainbowconnection}.
\end{proof}

\begin{thm}\label{scfc(K-s-t)}
For integers $s$ and $t$ with $1\leq s\leq t$, $\mathit{scfc}(K_{s,t})=\lceil\sqrt[s]{t}\rceil$.
\end{thm}
\begin{proof}
Since $diam(K_{s,t})=2$, from Theorem \ref{K-s-t} we have that $\mathit{scfc}(K_{s,t})=\lceil\sqrt[s]{t}\rceil$.
\end{proof}

\begin{pro}\label{C_n}
Let $C_n$ be a cycle of order $n$ and let $P_n$ be a spanning subgraph of $C_n$. Then $\mathit{scfc}(C_n)\leq \mathit{scfc}(P_n)$.
\end{pro}
\begin{proof}
Let $P_n= e_1e_2\cdots e_{n-1}$ be a path with $n$ vertices and let $u$ and $v$ be the ends of $P_n$. We know that $\mathit{scfc}(P_n)=\lceil\log_2 n\rceil$ by Theorem \ref{scfc tree}. Now we first give a coloring for $P_n$: color the edge $e_i$ with color $x+1$, where $2^x$ is the largest power of 2 that divides $i$. One can see that $\lceil\log_2 n\rceil$ is the largest number in the coloring by Theorem \ref{scfc tree}. Clearly, the color $\lceil\log_2 n\rceil$ only occurs once. Thus, we color the edge $uv$ with $\lceil\log_2 n\rceil$ in $C_n$ if there is only one color occurring once; otherwise, we color the edge $uv$ with $\lceil\log_2 n\rceil-1$. Consequently, the coloring is a strong conflict-free connection coloring of $C_n$.
\end{proof}
\textbf{Remark}: The proposition does not hold for general graphs. Here is a counterexample. Let $G= C_6$ with the edge set $\{v_1v_2,v_2v_3,v_3v_4,v_4v_5,v_5v_6,v_6v_1\}$. So $\mathit{scfc}(G)=2$. Let $G'=C_6+ v_1v_3$. Then $\mathit{scfc}(G')=3$.

\begin{thm}\label{cycle}
If $C_{n}$ is a cycle with $n$ $(n\geq3)$ vertices, then
\begin{center}
$\mathit{scfc}(C_{n})=\lceil\log_2 n\rceil-1$ or $\lceil\log_2 n\rceil$.
\end{center}
\end{thm}
\begin{proof}
By Proposition \ref{C_n} and Theorem \ref{scfc tree}, one can see that $\mathit{scfc}(C_{n})\leq \lceil\log_2 n\rceil$. It remains to handle with the lower bound. We first consider the case that $diam(C_n)=\frac{n}{2}$ for $n=2k$ $(k\in \mathbb{Z}^+)$. Hence, $\mathit{scfc}(C_n)\geq \lceil\log_2(\frac{n}{2}+1)\rceil=\lceil\log_2(n+2)\rceil-1\geq \lceil\log_2n\rceil-1$. We then consider the case that $diam(C_n)=\frac{n-1}{2}$ for $n=2k+1$ $(k\in \mathbb{Z}^+)$. Thus, $\mathit{scfc}(C_n)\geq \lceil\log_2(\frac{n-1}{2}+1)\rceil=\lceil\log_2(n+1)\rceil-1\geq \lceil\log_2n\rceil-1$. Consequently, $\mathit{scfc}(C_{n})=\lceil\log_2 n\rceil-1$ or $\lceil\log_2 n\rceil$.
\end{proof}

Theorem \ref{cycle} implies the following corollary.
\begin{cor}\label{onecycle}
Let $G$ be a connected graph with $m$ edges and let $C$ be a cycle in $G$. Then $\mathit{scfc}(G)\leq m-|C|+\lceil\log_2|C|\rceil$.
\end{cor}
\begin{proof}
By Theorem \ref{cycle}, $\mathit{scfc}(C)\leq \lceil\log_2|C|\rceil$. If we color the edges of $C$ with $\lceil\log_2|C|\rceil$ colors to make $C$ strongly conflict-free connected, and color each of the remaining $m-|C|$ edges with a fresh color, then we can verify that $G$ is strongly conflict-free connected. Consequently, $\mathit{scfc}(G)\leq m-|C|+\lceil\log_2|C|\rceil$.
\end{proof}

A graph $G$ is called $k$-$\mathit{cfc}$-$critical$ if $\mathit{cfc}(G)=k$ and for any proper subgraph $G'$ of $G$, $cfc(G')<k$.

\begin{thm}\label{Q_k}\upshape\cite{JLZ}
Let $Q_k$ be the graph obtained from two copies of $K_{1,k-1}$ with
$k\geq 2$ by identifying a leaf vertex in one copy with a leaf vertex in
the other copy. Then $Q_k$ is $k$-$\mathit{cfc}$-critical.
\end{thm}

\section{Upper and lower bounds}\label{The bound}

At first, let us look at trees.
\begin{thm}\label{diameter}
Let $T$ be a tree of order $n$. Then
\begin{center}
$\max\{\lceil\log_2(diam(T)+1)\rceil,\Delta(T)\}$$\leq \mathit{scfc}(T)\leq n-1$.
\end{center}
\end{thm}
\begin{proof}
Clearly, it is a strong conflict-free coloring that colors the edges of $T$ with distinct colors, and so the upper bound holds. For the lower bound, let $P$ be a path of length $diam(T)$, which needs at least $\lceil\log_2(diam(T)+1)\rceil$ colors by Theorem \ref{scfc tree}. Meanwhile, since a strong conflict-free connection coloring of a tree must be a proper edge-coloring, it is obvious that $\mathit{scfc}(T) \geq \chi'(T)= \Delta(T)$. Now we show that the upper bound is sharp. Let $H=(V(H),E(H))$ be a star. Then $\mathit{scfc}(H)=|E(H)|$ by Theorem \ref{scfc tree}. The lower bound is sharp by Theorem \ref{scfc tree}.
\end{proof}

Before we show the following theorem, we first define the notion of \emph{$t$-parallel paths}. Let $G$ be a connected graph and let $u$, $v$ be two vertices of $G$. If there are $t$ paths between $u$ and $v$ in $G$, where the degree of internal vertices of the paths is 2, then we call the paths \emph{$t$-parallel paths}.
\begin{thm}\label{scfc>2}
Let $G$ be a connected graph and let $v$, $u$ be two vertices of $G$ with $d(u,v)\geq2$. If one of the following conditions holds, then $\mathit{scfc}(G)\geq 3$.
\begin{enumerate}
  \item There exist a cut-vertex $w$ which splits $G$ into at least three components by deleting $w$.
  \item There exists a path $P$ of length at least $4$ between $u$ and $v$, where the edges of the path are bridges.
  \item There exist \emph{2-parallel paths} between $u$ and $v$, where the length of one path is $2$ and the length of the other one is $3$.
  \item There exist \emph{5-parallel paths} between $u$ and $v$.
\end{enumerate}

\end{thm}
\begin{proof}
1. Let $C_1,C_2,\cdots, C_m$ $(m\geq 3)$ be the components when deleting $v$ from $G$. We choose a vertex $u_i$ which is adjacent to $v$ in each component $C_i$. Clearly, each pair of $u_i$ and $u_j$ contains the only path, and it contains $v$. Consequently, the subgraph of $G$  induced by $\{u_1,u_2,\cdots,u_m,v\}$ is a star $S_m$. By Theorem \ref{scfc tree}, we have $\mathit{scfc}(S_m)=m$. Because every pair of $u_i$ and $u_j$ has the same path in $S_m$ and in $G$, we have that $\mathit{scfc}(G)\geq \mathit{scfc}(S_m)=m\geq 3$.

2. Let $P=v_1v_2\cdots v_t$ $(t\geq 4)$. Since every edge of $P$ is a bridge, each pair of vertices $v_i$, $v_j$ contain the same path in $P$ and in $G$. Hence, we have $\mathit{scfc}(G)\geq \mathit{scfc}(P)\geq3$.

3. Since the lengths of the two paths are 2 and 3, there is a 5-cycle in $G$. Clearly, $\mathit{scfc}(G)\geq 3$.

4. Since $d(u,v)\geq2$, every path between $u$ and $v$ has a length at least 2. If we assign a coloring with 2 colors for the paths, then there always exist at least two internal vertices of the paths which do not contain a strong conflict-free path. Consequently, $\mathit{scfc}(G)\geq 3$.
\end{proof}

We now define a graph class. Let $S_m$ be a star with $m$ leaves $v_1,v_2,\cdots,v_m$.  We denote  by $S_{m,t}$ the graph $S_m +\{v_{i_1}v_{i_2},v_{i_3}v_{i_4},\cdots,v_{i_{t-1}}v_{i_t}\}$ $(v_{i_j}\neq v_{i_k}| j,k\in[t], i_k,i_j\in [m])$.

\begin{thm}\label{G_t}
If $G_t$ is a connected graph with $m$ $(m\geq2)$ edges and $t$ edge-disjoint triangles, then
$\mathit{scfc}(G_t)\leq m-2t$,
and the equality holds if and only if $G_t\cong S_{m,t}$.
\end{thm}
\begin{proof}
Clearly, $\mathit{scfc}(K_3)=1$. Now we first give a coloring of $G_t$: color each triangle with a distinct color, that is, the three edges of each triangle receive a same color, and color each of the remaining $m-3t$ edges with a distinct color. Let $P$ be a strong conflict-free path for any pair of vertices $u$ and $v$ in $G$. Clearly, $P$ contains at most one edge from each triangle. Otherwise, it will produce a contradiction. Thus, $G_t$ is strongly conflict-free connected. So $\mathit{scfc}(G_t)\leq m-2t$.

We now show that the equality holds if and only if $G_t\cong S_{m,t}$.

\textbf{Claim 1.} $\mathit{scfc}(S_{m,t})=m-2t$.

\textbf{Proof of Claim 1.} Clearly, $\mathit{scfc}(S_{m,t})\leq m-2t$. It remains to show the other round. Note that every pendant edge needs a distinct color and every triangle needs a fresh color. Assume that we color some triangle with one color used on some pendant edge. Then the shortest path is not a conflict-free path between the leaf incident with the pendant edge and one vertex of degree two. Also, if we provide the $t$ triangles with $t-1$ colors, there exist two triangle with the same color. There would also not exist a strong conflict-free path between the vertices of the two triangles. Consequently, $\mathit{scfc}(S_{m,t})\geq m-2t$.

\textbf{Claim 2.} Every edge is a cut-edge except the edges of triangles.

\textbf{Proof of Claim 2.} Assume that there is a cycle $C$ $(|C|\geq 3)$ except the $t$ triangles. By Theorem \ref{cycle}, we know that $\mathit{scfc}(C)\leq \lceil\log_2{|C|}\rceil$. Now we provide a coloring: color every triangle with a distinct color and color $C$ with $\lceil\log_2{|C|}\rceil$ fresh colors, and the remaining edges are colored by $m-|E(C)|-3t$ fresh colors. Clearly, $G_t$ is strongly conflict-free connected. So, $\mathit{scfc}(G_t)\leq m-2t+\lceil\log_2{|C|}\rceil-|C|\leq m-2t-1$, a contradiction.

\textbf{Claim 3.} Each triangle contains at least two vertices of degree two in $G_t$.

\textbf{Proof of Claim 3.} Assume that there is only one vertex of degree two in a triangle $A$, say $A=v_1v_2v_3v_1$. Without loss of generality, let $u_1v_1$ and $u_2v_2$ be two edges. We will consider the following three cases:

\emph{Case 1.} Both $u_1v_1$ and $u_2v_2$ are not contained in triangles. In order to find out a contradiction, we provide a coloring $c$: assign each triangle with a distinct color; assign both $u_1v_1$ and $u_2v_2$ with a fresh same color; the remaining $m-2-3t$ edges are colored by $m-2-3t$ fresh colors. We only need to check $u_1$-$u_2$ paths. By Claim 2, there is no other cycle except the $t$ triangles. So $u_1v_1v_2u_2$ is the unique path which is strongly conflict-free connected. Clearly, $G_t$ is strongly conflict-free connected. Hence, $\mathit{scfc}(G_t)\leq (m-2-3t)+1+t=m-2t-1$, a contradiction.

\emph{Case 2.} $u_1v_1$ and $u_2v_2$ are contained in different triangles. Let $X_1$ contain $u_1v_1$ and let $X_2$ contain $u_2v_2$. We now provide a coloring: assign $X_1$ and $X_2$ with the same color; assign the other triangles with $t-2$ fresh colors; each of the remaining edges is colored by a fresh color. Clearly, $G_t$ is strongly conflict-free connected, a contradiction.

\emph{Case 3.} One of $u_1v_1$ and $u_2v_2$ is contained in a triangle. Without loss of generality, let $u_1v_1$ be contained in a triangle $X_3$. We color $X_3$ and $u_2v_2$ with the same color, the coloring of remaining edges is the same as Case 2. Also, this is a strong conflict-free connection coloring, a contradiction. Completing the proof of Claim 3.

Let $C(G_t)$ be the graph induced by all the cut-edges of $G_t$.

\textbf{Claim 4.} $C(G_t)$ is a tree.

\textbf{Proof of Claim 4.} Assume $C(G_t)$ is not connected. Let $H_1$ and $H_2$ be two components with $C(G_t)=H_1\cup H_2$. There exists one leaf $r_1$ in $H_1$ and one leaf $r_2$ in $H_2$ which are contained in the same triangle, say $r_1vr_2r_1$. Otherwise, $G_t$ is not connected. But both $d(r_1)\geq 3$ and $d(r_2)\geq3$, which contradicts Claim 3.

\textbf{Claim 5.} $diam(C(G_t))\leq 2$.

\textbf{Proof of Claim 5.} Assume that $diam(C(G_t))=k\geq3$. Let $P=v_0v_1\cdots v_k$ be a path of length $k$. Then we provide a coloring $c:$ $E(G)\mapsto [m-2t-k+\lceil\log_2(k+1)\rceil]$ of $G_t$: assign the edges of $P$ with $\lceil\log_2k\rceil$ colors to make $P$ strongly conflict-free connected by Theorem \ref{scfc tree}; assign each of the $t$ triangles with a fresh color; assign each of the remaining $m-3t-k$ edges with a fresh color. Clearly, $G_t$ is strongly conflict-free connected, a contradiction.

Clearly, from the above Claims we can deduce that $G_t\cong S_{m,t}$.
\end{proof}

\section{Graphs with large or small $\mathit{scfc}$ numbers}\label{Large and small}

In this section, we characterize the connected graphs $G$ of size $m$ with $\mathit{scfc}(G)=k$ for $k\in \{1,m-3,m-2,m-1,m\}$.

\begin{thm}\label{scfc=1}
For a nontrivial connected graph $G$,
$\mathit{scfc}(G)=1$ if and only if $G$ is a complete graph.
\end{thm}
\begin{proof}
Suppose that $G$ is a complete graph. Clearly, we have that $\mathit{scfc}(G)=1$. Conversely, suppose that $\mathit{scfc}(G)=1$. Assume that $G$ is not complete. Then there exists a pair of vertices $u,v$ with $d(u,v)\geq 2$. So, $\mathit{scfc}(G)\geq 2$, a contradiction. Thus, $G$ must be a complete graph.
\end{proof}

We now present an observation which will be used in the sequel.
\begin{ob}\label{observation}
Let $G$ be a connected graph with $\mathit{scfc}(G)=|E(G)|-k$ and let $H$ be a connected graph with $\mathit{scfc}(H)\leq |E(H)|-k-1$. Then $G$ does not contain a copy of $H$.
\end{ob}
\begin{proof}
Assume, to the contrary, that $G$ contains a copy of $H$. Then, we give  a coloring for $G$ as follows: assign the edges of $H$ with $|E(H)|-k-1$ colors to make $H$ strongly conflict-free connected, and then assign each of the remaining $m-|E(H)|$ edges of $G$ with a fresh color. Clearly, $G$ is strongly conflict-free connected. Consequently, $\mathit{scfc}(G)\leq E(G)-|E(H)|+|E(H)|-k-1\leq E(G)-k-1$, a contradiction.
\end{proof}

The following are two useful lemmas which will help to prove our latter theorems.
\begin{lem}\label{logdiam-diam}
Let $G$ be a connected graph with size $m$ and $\mathit{scfc}(G)= m-k$. Then
\begin{center}
$diam(G)-\lceil\log_2(diam(G)+1)\rceil\leq k$.
\end{center}
\end{lem}
\begin{proof}
Let $P$ be the path of length $diam(G)$. Now we provide a coloring with $m+\lceil \log_2{diam(G)+1}\rceil-diam(G)$ colors: assign the edges of $P$ with $\lceil \log_2{diam(G)+1}\rceil$ colors to make $P$ strongly conflict-free connected; assign each of the remaining $m-diam(G)$ edges a fresh color. Clearly, $G$ is strongly conflict-free connected. Since $\mathit{scfc}(G)= m-k$, then we have that $m-k\leq m+\lceil\log_2{(diam(G)+1)}-diam(G)\rceil$. Since $\lceil\log_2{(diam(G)+1)}-diam(G)\rceil$ is monotone decreasing, then $-diam(G)+\lceil\log_2{(diam(G_t)+1)}\rceil\geq -k$. Consequently, $diam(G)-\lceil\log_2(diam(G)+1)\rceil\leq k$.
\end{proof}

\begin{lem}\label{logC-C}
Let $G$ be a connected graph with size $m$ and $\mathit{scfc}(G)=m-k$, and let $C$ be a cycle of $G$. Then
\begin{center}
$|C|-\lceil\log_2|C|\rceil\leq k$.
\end{center}
\end{lem}
\begin{proof}
It is clear that $\mathit{scfc}(C)\leq \lceil\log_2|C|\rceil$ by Theorem \ref{cycle}. Then we give a coloring as follows: assign the edges of $C$ with $\lceil\log_2|C|\rceil$ colors to make $C$ strongly conflict-free connected and assign each of the remaining $m-|C|$ edges with a fresh color. We can easily verify that the coloring is a strong conflict-free coloring, a contradiction. Consequently, $|C|-\lceil\log_2|C|\rceil\leq k$.
\end{proof}

\begin{thm}\label{(3)}
Let $G$ be a nontrivial connected graph of size $m$. Then
$\mathit{scfc}(G)=m$ if and only if $G\cong S_m$.
\end{thm}
\begin{proof}
Suppose that $\mathit{scfc}(G)=m$. Assume that there is a cycle $C$ in $G$. Then $\mathit{scfc}(G)\leq m-|C|+\lceil\log_2|C|\rceil$ by Corollary \ref{onecycle}, which is a contradiction. Hence, $G$ is a tree. Let $u$ and $v$ be two vertices with $d_G(u,v)\geq3$ in $G$.  Assume that $P$ is a path of length $d_G(u,v)$ between $u$ and $v$. Then we provide a coloring for $G$: assign the edges of $P$ with $\lceil\log_2(d_G(u,v)+1)\rceil$ colors to make $P$ strongly conflict-free connected; assign each of the remaining edges with a fresh color. Clearly, $G$ is strongly conflict-free connected by the edge-coloring with $m-d_G(u,v)+\lceil\log_2(d_G(u,v)+1)\rceil$ colors, a contradiction. Thus, $G\cong S_m$.
\end{proof}

Before proving the theorem below, we define some graph-classes. Let $S_m$ be a star with $m \ (\geq 2)$ edges and let $u$ be a leaf of $S_m$. We define a graph by $\Gamma_{m+1}=(V(S)\cup\{v\},E(S)\cup\{uv\})$ and we denote by $P_n$ a path of length $n$.

\begin{lem}\label{gamma}
If $G\in\{P_3, P_4, \Gamma_m\}$, then $\mathit{scfc}(G)=m-1$.
\end{lem}
\begin{proof}
It is clear that $\mathit{scfc}(P_3)=2$ and $\mathit{scfc}(P_4)=3$ by Theorem \ref{scfc tree}, and $\mathit{scfc}(\Gamma_m)\geq\Delta(\Gamma_m)=m-1$ by Theorem \ref{diameter}. Then for the upper bound we give a coloring: assign each of the $m-1$ edges of $S_{m-1}$ with a fresh color and choose one color from the used colors except the color assigned to the edge incident with $u$. Clearly, $G$ is strongly conflict-free connected. Consequently, $\mathit{scfc}(\Gamma_m)=m-1$.
\end{proof}

\begin{thm}\label{(2)}
Let $G$ be a connected graph of size $m$. Then
$\mathit{scfc}(G)=m-1$ if and only if $G\in$ $\{P_3$, $P_4$, $\Gamma_m\}$.
\end{thm}
\begin{proof}
 The necessity holds by Lemma \ref{gamma}. On the contrary, suppose that $\mathit{scfc}(G)=m-1$. We first claim that $G$ is a tree. Assume that $G$ is not a tree. Let $C$ $(|C|\geq 3)$ be a cycle of $G$. We have that $\mathit{scfc}(C)\leq |C|-2$ by Corollary \ref{onecycle}, and it is not true by Observation \ref{observation}.

Suppose that $diam(G)\geq 5$ in $G$. Clearly, $diam(G)-\lceil\log_2(diam(G)+1)\rceil>1$ by Lemma \ref{logdiam-diam}, a contradiction. So $diam(G)\leq 4$. Suppose $diam(G)=4$. Let $P_4=v_1v_2v_3v_4v_5$ be a path with $\mathit{scfc}(P_4)=3$. If $G=P_4$, then it is true. Assume that there is another vertex $w$ adjacent to $v_i$ of $P_4$, denote this structure by $R$. It is clear to see that $R$ can be colored by three colors to make it strongly conflict-free connected. Thus, $\mathit{scfc}(R)=|E(R)|-2$, and $R\nsubseteq G$ by Observation \ref{observation}. Consequently, $G\cong P_4$.

Suppose that $diam(G)=3$. Let $P_3=v_1v_2v_3v_4$ be a path with $\mathit{scfc}(P_3)=2$. If $G=P_3$, then it is true. Assume that there are two vertices $x,y$ adjacent to $v_2$, $v_3$ of $P_3$, respectively, denote by $L$ this structure. It is easy to check that $\mathit{scfc}(L)\leq 3$. So $L\nsubseteq G$ by Observation \ref{observation}. Without loss of generality, let $d(v_2)=t$ $(\geq 3)$ and $d(v_3)=2$. Obversely, $\mathit{scfc}(G)\geq t$ by Theorem \ref{diameter}. Now we assign each of the edges incident with $v_2$ by a fresh color and assign the remaining edge $e$ by the color used on some edge not adjacent to $e$. Clearly, $G$ is strongly conflict-free connected. So, $G\in\{P_3,\Gamma_m\}$. Suppose that $diam(G)=2$. Then $G\cong S_n$ with $\mathit{scfc}(G)=m$, a contradiction. Completing the proof.
\end{proof}

\begin{figure}[!htb]
\centering
\includegraphics[width=0.5\textwidth]{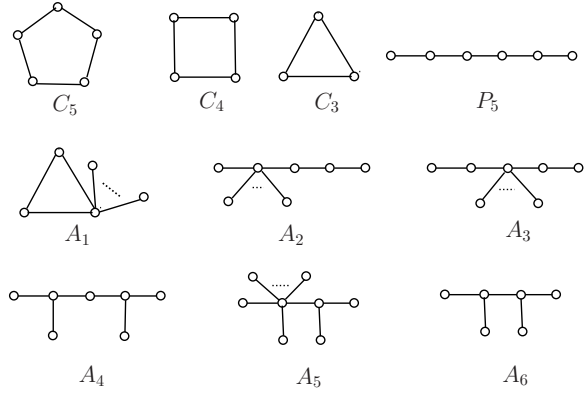}
\caption{Graphs with $\mathit{scfc}(G)=m-2$ }\label{1}
\end{figure}
\begin{thm}\label{thmm-2}
$\mathit{scfc}(G)=m-2$ if and only if $G\in\{C_3,C_4,C_5,P_5,A_1,A_2,\cdots,A_6\}$.
\end{thm}
\begin{proof}
Suppose that $\mathit{scfc}(G)=m-2$. Then $diam(G)\leq 5$ by Lemma \ref{logdiam-diam}. Let $C$ be a cycle of $G$. Then $|C|\leq 5$ by Lemma \ref{logC-C}. We now distinguish the following cases (The graphs are demonstrated in Figure \ref{1}).

\emph{Case 1}: $G$ is a tree.

$(i)$  Suppose that $diam(G)=5$. Let $P_5=v_1v_2v_3v_4v_5v_6$ be a path. If $G=P_5$, then it is true by Theorem \ref{scfc tree}. If $G\neq P_5$, then we construct a graph $H$ by adding an edge $uv_i$ $(i\in\{2,3,4,5\})$ to $P_5$. It is easy to get $\mathit{scfc}(H)\leq 3$. Hence, $H\nsubseteq G$ by Observation \ref{observation}. Consequently, $G\cong P_5$.

$(ii)$  Suppose that $diam(G)=4$. Let $P_4=v_1v_2v_3v_4v_5$ be a path for which $\mathit{scfc}(P_4)=3$, and thus, $G\neq P_4$. We construct a graph $H_1$ by adding two vertices $u_1$, $u_2$ and connecting them to $v_2$ and $v_3$, respectively. Then $H_1\nsubseteq G$ by Observation \ref{observation} since $\mathit{scfc}(H_1)=|E(H_1)|-3$. We construct another graph $H_2$ by adding a vertex $u_1$ and connecting it to $v_2$ and adding another vertex $u_2$ and connecting it to $v_4$, which means that $\mathit{scfc}(H_2)=|E(H_2)|-2$. We construct $H_2'$ by adding an edge $w_1v_2$ in $P_4$. We have that $\mathit{scfc}(H_2')=|E(H_2')|-2$ by easy calculation. Let $P=u_1u_2u_3$ be a path. We construct a graph $H_3$ by identifying $u_1$ with $v_3$ of $P_4$ (if $u_1$ is identified with other vertices of $P_4$, then it contradicts that $diam(G)=4$). Clearly, $\mathit{scfc}(H_3)=|E(H_3)|-3$. Thus, $H_3\nsubseteq G$ by Observation \ref{observation}. Let $x_1x_2$ be an edge and we construct $H_4$ by identifying $x_1$ with $v_3$ of $H_2$. Clearly, $\mathit{scfc}(H_4)=|E(H_4)|-3$. Thus, $H_4\nsubseteq G$ by Observation \ref{observation}. Consequently, $G$ can contain $H_2$ and $H_2'$ but not $H_3$ and $H_4$.

Now we show that $G\in \{A_2,A_3,A_4\}$. Clearly, it is true for $G=A_4=H_2$. $A_2$ is constructed by identifying one end of each of $l$ new edges with $v_2$ in $P_4$. Clearly, $\mathit{scfc}(A_2)\geq \Delta(A_2)$ by Theorem \ref{diameter}. Then we give a coloring of $A_2$: first, assign each of the $\Delta(A_2)$ edges incident with $v_2$ by a fresh color, and assign the remaining two edges with used colors except the color used on $v_2v_3$. Clearly, it is a strong conflict-free connection coloring. So $\mathit{scfc}(A_2)=|E(A_2)|-2$. Similarly, $\mathit{scfc}(A_3)=|E(A_3)|-2$. Consequently, $G\in \{A_2,A_3,A_4\}$.

$(iii)$  Suppose that $diam(G)=3$. By above similar manner, we have that $\mathit{scfc}(G)=m-2$ if and only if $G\in\{A_5,A_6\}$.
Suppose $diam(G)=2$. Then the only graph is a star $S_m$, which is a contradiction with $\mathit{scfc}(S_m)=m$.

\emph{Case 2}: There is at least one cycle $C$ in $G$.

$(i)$ Suppose that $C=v_1v_2v_3v_4v_5v_1$. If $G=C$, then it is true by $\mathit{scfc}(C)=3$. Assume that there is an edge $v_iv_j$ $(i,j\in[5])$ in $G$. Then there is a subgraph of $G$, say $C_5'=C+v_iv_j$. Then $C_5'\nsubseteq G$ by Observation \ref{observation} since $\mathit{scfc}(C_5')=3$. We construct the graph $C^+_5$ by adding $u$ and connecting it to vertex $v_i$ $(i\in[5])$ of $C_5$. Obviously, $C^+_5\nsubseteq G$ by Observation \ref{observation} since $\mathit{scfc}(C^+_5)=3$. So $G\cong C_5$.

$(ii)$  Suppose that $C=v_1v_2v_3v_4v_1$. We construct the graph $C_4'$ by adding one edge $v_2v_3$ in $C_4$. Then $C_4'\nsubseteq G$ by Observation \ref{observation} since $\mathit{scfc}(C_4')=2$. We then construct another graph $C^+_4$ by adding one vertex $w$ and connecting it to $v_i$ $(i\in[4])$. Obviously, $C^+_4\nsubseteq G$. Consequently, $G\cong C_4$.

$(iii)$ \ Suppose that $C=v_1v_2v_3v_1$. Let $P=u_1u_2u_3$ be a path. We construct a graph by identifying $u_1$ with $v_i$ $(i\in[3])$, denote it by $H_1$. Clearly, $H_1\nsubseteq G$ by Observation \ref{observation} since $\mathit{scfc}(H_1)\leq2$. Let $P'=u_1u_2$ and $P''=w_1w_2$ be two paths. We construct the graph $H_2$ by identifying $u_1$ with $v_i$ and identifying $w_1$ with $v_j$ $(i,j\in[3]$ and $i\neq j)$. Clearly, $\mathit{scfc}(H_2)=2$. So $H_1\nsubseteq G$ by Observation \ref{observation}. Finally we construct $H_3$ by identifying one vertex $v_i$ with $u_1$ of $P'$. Clearly, $\mathit{scfc}(H_3)=2$. Therefore, $G$ does not contain $H_1$ and $H_2$ but $G$ contain $H_3$. Clearly, the only graph class must be $A_1$ and $\mathit{scfc}(A_1)=|E(A_1)|-2$. Consequently, $G\in\{C_3,A_1\}$.
\end{proof}

\begin{figure}[!htb]
\centering
\includegraphics[width=0.5\textwidth]{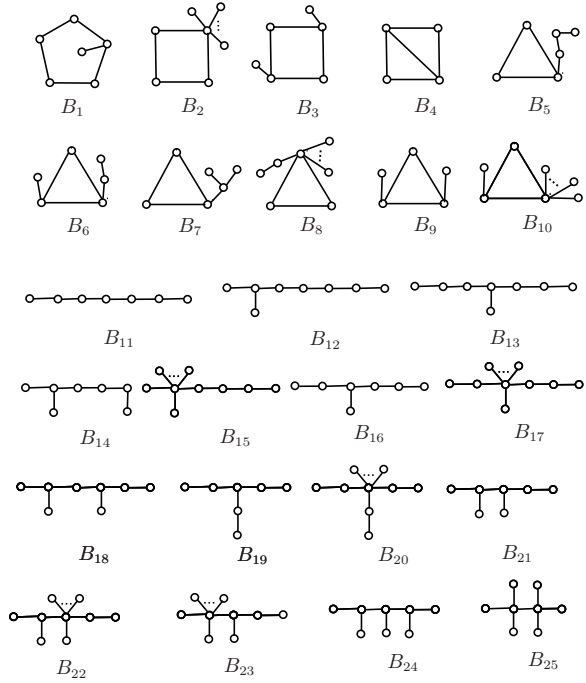}
\caption{Graphs with $\mathit{scfc}(G)=m-3$ }\label{2}
\end{figure}

\begin{lem}\label{lemmam-3}
Let $G$ be a connected graph of size $m$. If $G\in\{B_1,B_2,\cdots,B_{25}\}$, then $\mathit{scfc}(G)=m-3$.
\end{lem}
\begin{proof}
The graphs are demonstrated in Figure \ref{2}. For $G\in\{B_1,\cdots,B_7,B_9,B_{11},$ $\cdots,B_{14},B_{17},B_{18},B_{19},B_{21},B_{23},B_{24},B_{25}\}$, we can easily check that $\mathit{scfc}(G)=m-3$. For $G\in\{B_8,B_{10}\}$, there is a triangle in $B_8$ and $B_{10}$, respectively. Then clearly $\mathit{scfc}(G)\geq \Delta(G)-1$. We give a coloring for $B_8$: color the triangle by 1 and color each of the edges incident with the vertices of the triangle by a fresh color, and color the remaining edge by a color used on the edges not adjacent to it. Clearly, $\mathit{scfc}(B_8)=|E(B_8)|-3$. Similarly, $\mathit{scfc}(B_{10})=|E(B_{10})|-3$. For $B_{15}$, it can be obtained by identifying one leaf of $P_{3}$ with one leaf of $S_t$ $(t\geq 3)$. Then we have that $\mathit{scfc}(S_t)\geq t$. We give a coloring of $B_{15}$: color the edges of $S_t$ with $t$ colors and choose two colors used on two leaves of $S_t$ to color the remaining two edges. Clearly, it is a strong conflict-free connection coloring. Thus, $\mathit{scfc}(B_{15})=|E(B_{15})|-3$. Similarly, we can easily check that $\mathit{scfc}(G)=m-3$ for $G\in\{B_{20},B_{22},B_{23},B_{25}\}$.
\end{proof}

\begin{thm}\label{thmm-3}
Let $G$ be a connected graph with $m \ (m\geq4)$ edges. Then $\mathit{scfc}(G)=m-3$ if and only if $G\in\{B_1,B_2,\cdots,B_{25}\}$.
\end{thm}
\begin{proof}
The sufficiency holds by Lemma \ref{lemmam-3}.
Now we consider the necessity. We first have $diam(G)\leq 6$ by Lemma \ref{logdiam-diam}. Let $C$ be a cycle in $G$. Then $|C|\leq 6$ by Lemma \ref{logC-C}. Then we consider the following two cases.

\textbf{Case 1.} Suppose that there is at least one cycle $C$ in $G$.

\emph{(i)}  Assume that $|C|=6$. Then $C\nsubseteq G$ by Observation \ref{observation} since $\mathit{scfc}(C)=2$. Thus, $|C|\leq 5$.

\emph{(ii)}  Assume that $|C|=5$. Let $C'$ be a graph by adding a chord to $C$. We can easily check that $\mathit{scfc}(C')=|E(C')|-4$. So, $C'\nsubseteq G$ by Observation \ref{observation}. We construct $C''$ by adding a leaf vertex to $C$, for which $\mathit{scfc}(C'')=|E(C'')|-3$. Then we construct $C'''$ by adding two leave vertices to $C$. But, $\mathit{scfc}(C''')=|E(C''')|-4$. Let $P$ be a path of length 2. We construct $\bar{C}$ by identifying a  vertex of $C$ with an end of $P$, for which $\mathit{scfc}(\bar{C})=|E(\bar{C})|-4$. Consequently, $G\cong C''=B_1$.

\emph{(iii)} \ Suppose that $C=v_1v_2v_3v_4v_1$. Let $P=u_1u_2u_3$ be a path. Then we construct $H_1$ by identifying $u_1$ in $P$ with $v_4$ in $C$. Clearly, $\mathit{scfc}(H_1)\leq2=|E(H_1)|-4$. Hence, $H_1\nsubseteq G$ by Observation \ref{observation}. Let $w_1w_2$ be an edge. We then construct $H_2$ by choosing arbitrarily $v_i$ $(i\in[4])$ and identifying $v_i$ with $w_1$. Then we get $\mathit{scfc}(H_2)=|E(H_2)|-3$. Let $s_1s_2$ be an edge. We construct $H_3$ by identifying $v_1$, $v_2$ of $C_4$ with $s_1$, $w_1$, respectively. Then clearly $\mathit{scfc}(H_3)=|E(H_3)|-4$=2, and thus $H_3\nsubseteq G$. We construct $H_4$ by identifying $s_1$ and $w_1$ with $v_1$ and $v_3$ of $C$, respectively. Clearly, $\mathit{scfc}(H_4)=|E(H_4)|-3$.

Hence, $G$ can contain the copies of $H_2$ and $H_4$ but not the copies of $H_1$ and $H_3$. Clearly, suppose that we construct $H_5$ by adding one pendant vertex to $v_1$ in $H_4$. Clearly, $\mathit{scfc}(H_5)=|E(H_5)|-4$. Then it does not hold for $H_5$. Obviously, $G\in\{B_2,B_3\}$. We construct $H_6$ by adding a chord to $C_4$. Then we have that $\mathit{scfc}(H_5)=|E(H_5)|-3$. At last, we construct $H_6$ by adding a leaf vertex to connect it to a vertex of $H_5$. Clearly, $\mathit{scfc}(H_6)=|E(H_6)|-4$. Consequently, $G\in\{B_2,B_3,B_4\}$.

\emph{(iv)}  Suppose that $C=v_1v_2v_3v_1$. Let $P=u_1u_2u_3u_4u_5$. Clearly, $\mathit{scfc}(P)=3$ and $\mathit{scfc}(C)=1$. We construct a graph $H_1$ by identifying $u_1$ with $v_1$. Clearly, the coloring by assigning each edge $e$ $\in E(C)$ and $u_3u_4$ with color 2 and assigning $v_1u_2$ and $u_4u_5$ with color 1 and assigning $u_2u_3$ with color 3 is a strong conflict-free connection coloring. So $\mathit{scfc}(H_1)\leq 3$. By Observation \ref{observation} $G$ does not contain any copy of $H_1$. Then we can use Observation \ref{observation} repeatedly, and eventually get that $G\in \{B_5,B_6,B_7,B_8,B_9,B_{10}\}$.

\textbf{Case 2.} Suppose that $G$ is a tree. By the same arguments, we know that $G\in\{B_{11},B_{12},B_{13}\}$ if $diam(G)=6$; $G\in \{B_{14},\cdots,B_{18}\}$ if $diam(G)=5$; $G\in \{B_{19},\cdots,B_{24}\}$ if $diam(G)=4$; $G=B_{25}$ if $diam(G)=3$. But $\mathit{scfc}(G)\leq m-1$ when $diam(G)\leq 2$ by Theorems \ref{(2)} and  \ref{(3)}.
\end{proof}

\section{ Cubic graphs with $\mathit{scfc}$-number 2}\label{cubic graphs}

In this section, we will characterize the cubic graphs $G$ with $\mathit{scfc}(G)=2$. We first discuss the relation between the strong conflict-free connection number and the strong proper connection number for cubic graphs.

We need the following definition.
\begin{df}
A \emph{forced 2-path} in a graph $G$ is a path $xyz$ such that $xz\notin E(G)$ and $xyz$ is the unique 2-path connecting $x$ and $z$.
A \emph{$k$-path} $P=u_0u_1\cdots u_k$ in a graph $G$ is called forced, if each 2-path $u_iu_{i+1}u_{i+2}$ is forced and $P$ is a path between $u_0$ and $u_k$, for $i=0,1,\cdots,k-2$.
A cycle of a graph $G$ is called a \emph{forced cycle} if any two successive edges of the cycle form a forced 2-path in $G$.
An edge $e$ in a graph $G$ is called a \emph{forced edge} if $e$ is not included in a cycle of length at most 4.
\end{df}

If $uv$ is a forced edge in $G$ and $vw$ is an edge adjacent to $uv$, then $uvw$ is a forced 2-path in $G$. The following two results follow directly from the definition.
\begin{lem}\label{path with forced}
Let $P=u_1u_2\cdots u_k$ be a forced path in $G$ with $\mathit{scfc}(G)=2$. Then the adjacent edges of $P$ are colored by distinct colors for every strong conflict-free connection coloring with 2 colors.
\end{lem}

\begin{lem}\label{cycle with forced}
Let $C=u_1u_2\cdots u_ku_1$ be a forced cycle of length $k$ in $G$ with $\mathit{scfc}(G)=2$. Then the adjacent edges of $C$ are colored by distinct colors for every strong conflict-free connection coloring with 2 colors and $k$ is even.
\end{lem}

Now we define some graph-classes. A $k$-ladder, denoted by $L_k$, is defined to be the product graph $P_k\Box K_2$, where $P_k$ is the path on $k$ vertices $(see \ Figure \ \ref{ladder})$. The M\"{o}bius ladder $M_{2k}$ is the graph obtained from $L_k$ by adding two new edges $s_1t_k$ and $t_1s_k$ $(see \ Figure \ \ref{mobius})$.

\begin{figure}[!htb]
\centering
\includegraphics[width=0.3\textwidth]{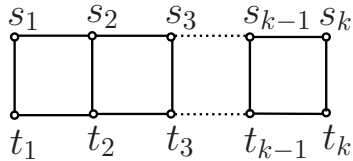}
\caption{The $k$-ladder $L_k$ }\label{ladder}
\end{figure}

\begin{lem}\label{scfc(C_kbox K_2}
$\mathit{scfc}(C_k\Box K_2)=2$ if and only if $k$ equals $3$, $4$ or $6$.
\end{lem}
\begin{proof}
Let $k\geq4$. Clearly, The graph has a forced cycle. Since $\mathit{scfc}(C_k\Box K_2)=2$, we have that $k=4$ or $6$ by Lemma \ref{cycle with forced}. When $k=3$, we define a $2$-edge-coloring $c$: for every edge $e$ in the triangles, $c(e)=1$; for the remaining edges $e$, $c(e)=2$. Clearly, the coloring is a strong conflict-free connection coloring for $C_3\Box K_2$. When $k=4,6$, we define a $2$-edge-coloring: assign alternate colors on the edges of $s_1s_2\cdots s_ks_1$ and $t_1t_2\cdots t_kt_1$ with colors 1 and 2 such that $c(s_1s_2)\neq c(t_1t_2)$, and all the remaining edges are colored by 1. One can easily check that this coloring is a strong conflict-free connection coloring.
\end{proof}

\begin{lem}\label{scfc(M{2k})=2}
$\mathit{scfc}(M_{2k})=2$ if and only if $3\leq k\leq7$.
\end{lem}

\begin{proof}
It is clear to see that $\mathit{scfc}(M_{2k})\geq2$ for every $k\geq3$ since $M_{2k}$ is not a complete graph. First, when $k\geq8$, clearly for the pair of vertices $s_2$ and $s_6$ there is only one shortest path connecting them, which is $P'=s_2s_3s_4s_5s_6$. For every pair of vertices in $P$, there is only one shortest path in $M_{2k}$ connecting them. So we have that $\mathit{scfc}(M_{2k})\geq \mathit{scfc}(P')=3$. For the graph $M_{2k}$ with $k\in\{4,6\}$, we define a $2$-edge-coloring $c$: for $i\in\{1,3,5\}$, $c(s_is_{i+1})=c(t_it_{i+1})=c(s_it_i)=1$; for the remaining edges $e$, $c(e)=2$. For the graph $M_{2k}$ with $k\in\{3,5,7\}$, we define a $2$-edge-coloring $c$: for $i\in\{1,3,5\}$, $c(s_is_{i+1})=c(t_{i+1}t_{i+2})=1$; for $i\in\{1,2\cdots,k\}$, $c(s_it_i)=c(s_kt_1)=1$; for the remaining edges $e$, $c(e)=2$.
It is easy to check that every pair of vertices are connected by a strong conflict-free path under the above $2$-edge-colorings.
\end{proof}
\begin{figure}[!htb]
\centering
\includegraphics[width=0.3\textwidth]{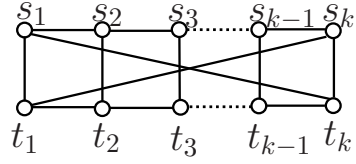}
\caption{The M\"{o}bius $M_{2k}$}\label{mobius}
\end{figure}

In order to be more convenient to handle with the following theorem, let us start with some explanations. Let $G$ be a cubic graph, and let $c:$ $E(G)\mapsto \{1,2\}$ be a strong conflict-free connection coloring of $G$. Let $P=(u=)v_1v_2\cdots v_{t-1}v_t(=v)$ be a strong conflict-free path between $u$ and $v$. Suppose that there exists a 2-path $v_iv_{i+1}v_{i+2}$ in $P$ such that $c(v_iv_{i+1})= c(v_{i+1}v_{i+2})$. Then there must exist another 2-path $v_iv_{i+1}'v_{i+2}$ with $c(v_iv_{i+1}')\neq c(v_{i+1}'v_{i+2})$ to replace $v_iv_{i+1}v_{i+2}$ since there exists a strong conflict-free path for the pair of $v_i$ and $v_{i+2}$. Then $v_iv_{i+1}'v_{i+2}$ is called\emph{ a replacement}. Furthermore, suppose that $c(v_{i-1}v_i)=c(v_iv_{i+1}')$. Then there must also exist a replacement $v_{i-1}v_i'v_{i+1}'$ with $c(v_{i-1}v_i')\neq c(v_i'v_{i+1}')$ for $v_{i-1}v_iv_{i+1}'$. Continue the operation. If there does not exist a replacement sharing the same edges with $P$, then the sequence of replacements is called a \emph{finite replacement} of $P$. Otherwise, the the sequence of replacements is called an \emph{infinite replacement} of $P$.

\begin{figure}[!htb]
\centering
\includegraphics[width=0.2\textwidth]{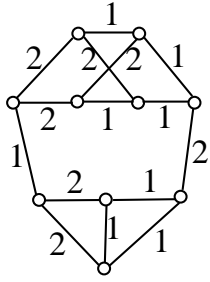}
\caption{The graph $U$ }\label{U}
\end{figure}

\begin{thm}\label{scfc=spc=2}
Let $G$ be a cubic graph with $G\ncong U$. If $\mathit{scfc}(G)=2$, then $\mathit{spc}(G)=2$.
\end{thm}

\begin{proof}
Let $c:$ $E(G)\mapsto [2]$ be a strong conflict-free connection coloring of $G$. Let $P=(u=)v_1v_2\cdots v_{t-1}v_t(=v)$ be an arbitrary strong conflict-free path between $u$ and $v$. For every pair of $v_i$ and $v_{i+2}$ $(i\in[t])$, if $c(v_iv_{i+1})\neq c(v_{i+1}v_{i+2})$, then $P$ is a strong proper path. Suppose that there exists $v_iv_{i+1}v_{i+2}$ $(i\in[t-2])$ in $P$ such that $c(v_iv_{i+1})= c(v_{i+1}v_{i+2})$. If there exist a finite replacement for $P$, then there is a strong proper path for every pair of vertices in $G$. Suppose that the replacement is an infinite one for $P$ $(see \ Figure\  \ref{scfc=spc})$.

We denote $G[V']$ by $W$, where $V'=\{v_{i-3},v_{i-2},v_{i-1},v_{i},v_{i+1},v_{i+2},v'_{i},v'_{i+1}\}$, and we say that $W$ is an \emph{attachment} of path $P$. Then we first show the following claims.

\begin{figure}[!htb]
\centering
\includegraphics[width=0.5\textwidth]{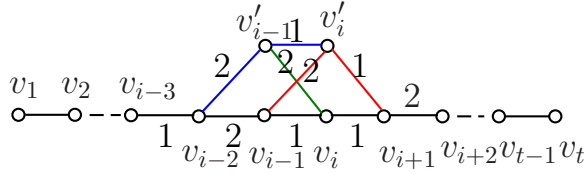}
\caption{The path $P$ with attachment $W$}\label{scfc=spc}
\end{figure}

\emph{Claim 1. For every strong conflict-free connection coloring $c$, $c(v_{i-3}v_{i-2})\neq c(v_{i-2}v_{i-1})=c(v_{i-2}v'_{i-1})$ and $c(v_{i+1}v_{i+2})\neq c(v_{i}v_{i+1})=c(v'_{i}v_{i+1})$.}

\emph{Proof of Claim 1:}  Without loss of generality, suppose that $c(v_{i-3}v_{i-2})= c(v_{i-2}v_{i-1})$. Then $v_{i_3}v_{i-2}v'_{i-1}$ is a unique shortest path between $v_{i-3}$ and $v'_{i-1}$
since $G$ is a cubic graph with $d(v'_{i-1})=3$. It contradicts that $c(v_{i-3}v_{i-2})= c(v_{i-2}v_{i-1})$ for the coloring $c$.

\emph{Claim 2. There is at most one attachment in $P$. Furthermore, let $C$ be a cycle. Then there are at most two attachments in $C$.}

\emph{Proof of Claim 2:} Assume that there are two attachments in $P$. Since $P$ is a shortest path, every subpath of $P$ is shortest. Hence, there is no strong conflict-free path between the attachments by Claim 1. Suppose that there are three attachments in $C$. Then $|C|\geq 12$, a contradiction by Claim 1. Completing the proof of Claim 2.

If the path $P$ with an attachment is not contained in a cycle, then there exist at least two cut-edges since $G$ is a cubic graph. Clearly, $\mathit{scfc}(G)\geq3$ by Claim 1. If we identify $v_{i-3}$ with $v_{i+2}$, then $G=M_6$ with $\mathit{spc}(M_6)=2$ by Lemma \ref{withoutforced}. Now we handle with the case that $P$ with an attachment is contained in a shortest cycle $C$. Clearly, $|C|\geq 6$, otherwise, $P$ does not contain an attachment. Suppose $|C|=6$. Then there are two vertices $u_1,u_2$ except the vertices of the attachment in $C$. If $u_1$ and $u_2$ are not adjacent to the same neighbor, then every pair of edges incident with $u_1$ is a forced 2-path. Hence, there need at least three colors, a contradiction. Let $x$ be a common neighbor of $u_1$ and $u_2$, where $u_2$ is adjacent to $v_{i+1}$. Let $y$ be a neighbor of $x$, and $z$ be another neighbor of $y$ except $x$. Thus, $v_{i+1}u_2xyz$ is a unique forced path for the pair $v_{i+1},z$. Then it is not a strong conflict-free path by Lemma \ref{path with forced}. Suppose $|C|=7$.
Let $u_1,u_2,u_3$ be three vertices except the vertices of the attachment in $C$. If each of $u_1,u_2,u_3$ is in a triangle, then $G\cong U$ ($see \ Figure \ \ref{U}$). If one of $u_1,u_2,u_3$ is in a triangle, then there exists a unique forced 4-path for a pair of vertices in $C$, a contradiction. Suppose that $C=v_1v_2v_3v_4u_1u_2u_3u_4v_1$, and suppose further that there are two attachments in $C$. Then $G \cong L_2$ ($see \ Figure \ \ref{L}$) with an edge-coloring such that $\mathit{scfc}(G)=\mathit{spc}(G)=2$. Suppose that $u_1, u_2, u_3, u_4$ are in triangles. Then $G\cong L_3$ ($see \ Figure \ \ref{L}$) such that $\mathit{scfc}(G)=\mathit{spc}(G)=2$. Otherwise, there will exist a unique forced 4-path for a pair of vertices in $C$, a contradiction. Suppose that at most one triangle contains two of the vertices $u_1,u_2,u_3,u_4$, without loss of generality, say $u_1,u_2$. Suppose further that $u_3u_4$ is a forced edge. Then $c(v_1u_4)\neq c(u_4u_3)\neq c(u_4x)$, where $x$ is a neighbor of $u_4$ except $v_1,u_3$, a contradiction. Then suppose that $u_3, u_4$ are contained a 4-cycle $C'$. Clearly, there provides a unique forced 4-path for the pair of $v_2$ and one vertex in $C'$ except $u_3,u_4$, a contradiction.
Suppose $9\leq|C|\leq 10$. Then there is a unique forced 4-path for some pair of vertices in $G$. Hence, $\mathit{scfc}(G)\geq 3$, a contradiction. Assume $|C|\geq 11$. Then there exists a unique shortest path of length 5 between $v_{i-3}$ and $v_{i+2}$, a contradiction by Claim 1.
\end{proof}

Now we only need to check whether 2 is the strong conflict-free connection number of $G$ with $\mathit{spc}(G)=2$ by Theorem \ref{scfc=spc=2}.

\begin{thm}\label{withoutforced}\upshape\cite{HY}
Let $G$ be a cubic graph without forced edges. Suppose further that $G\neq K_4$. Then $\mathit{spc}(G)=2$ if and only if $G\in\{C_3\Box K_2, C_{2k}\Box K_2, M_{2k}\}$ for some $k\geq 2$.
\end{thm}

Combining Theorem \ref{scfc=spc=2}, Theorem \ref{withoutforced}, Lemma \ref{scfc(C_kbox K_2} and Lemma \ref{scfc(M{2k})=2}, we have the following result.
\begin{lem}\label{scfcwithout}
Let $G$ be a cubic graph without forced edges. Then $\mathit{scfc}(G)=2$ if and only if $G\in\{C_{l}\Box K_2, M_{2k}\}$ for $l\in\{3,4,6\}$ and for $k$ with $3\leq k\leq7$.
\end{lem}

Let $F_0(k)$ be the cubic graph which is obtained from $L_k$ by adding two new vertices $x$ and $y$ and adding five new edges $xy,xs_1,xt_1,ys_k,yt_k$ ($see \ Figure \ \ref{F0(k)}$).

\begin{figure}[!htb]
\centering
\includegraphics[width=0.4\textwidth]{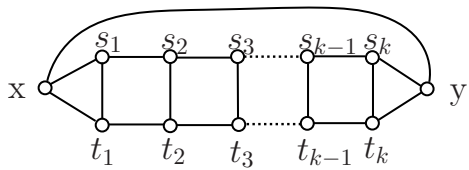}
\caption{The graph $F_0(k)$ }\label{F0(k)}
\end{figure}

\begin{lem}\label{scfcF0(k)=2}
$\mathit{scfc}(F_0(k))=2$ with $k\geq2$ if and only if $k\in \{2,4\}$.
\end{lem}
\begin{proof}
When $k\geq3$, the cycle $xs_1s_2\cdots s_ky$, say $C$, is a forced one in $F_0(k)$. Then we have that $k=4$ by Lemma \ref{cycle with forced}. When $k=2$, we define an edge-coloring $c$ for $F_0(k)$: $c(xy)=2$; $c(xs_1)=c(xt_1)=c(ys_k)=c(yt_k)=1$; $c(s_is_{i+1})=c(t_it_{i+1})=c(s_it_i)=1$ for even $i\in[k]$; for all the remaining edges, $c(s_is_{i+1})=c(t_it_{i+1})=c(s_it_i)=2$ for odd $i\in[k]$. We can  easily
check that every pair of vertices have a strong conflict-free path connecting them. Since $F_0(k)>1$, we have that $\mathit{scfc}(F_0(k))=2$ for $k=2$ or $4$.
\end{proof}

We now introduce a family $\mathcal{H}$ of graphs which are demonstrated in Figure \ref{H}.
\begin{center}
$\mathcal{H}=\{F_0^*(k),\hat{K_4},\hat{D_3},\tilde{K_{3,3}},\tilde{Q_3},F_1(k)\}$ $(k\in \mathbb{N})$
\end{center}

\begin{figure}[!htb]
\centering
\includegraphics[width=0.5\textwidth]{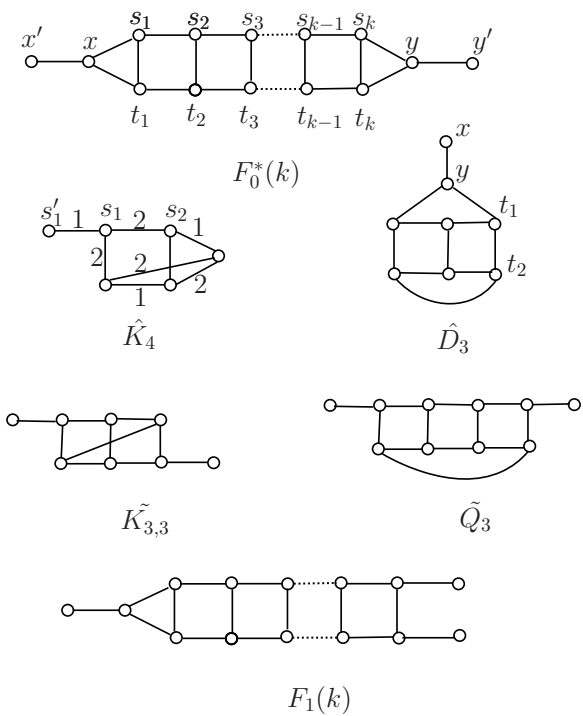}
\caption{The family of graphs $\mathcal{H}$ }\label{H}
\end{figure}

\begin{thm}\label{oneforcededge}\upshape\cite{HY}
Let $G$ be a cubic graph with exactly one forced edge. Then $\mathit{spc}(G)=2$ if and only if $G=F_0(k)$ for some even $k\geq4$, or $G$ is obtained from $H_1$ and $H_2$ by identifying the pendent edges to a single edge, where $H_i\in\{\hat{K_4},\hat{D_3}\}$ for $i=1,2$.
\end{thm}

\begin{lem}\label{singleedge}
Let $G$ be a cubic graph. If $G$ is obtained from $H_1$ and $H_2$ by identifying the pendent edges to a single edge, where $H_i\in\{\hat{K_4},\hat{D_3}\}$ for $i=1,2$, then $\mathit{scfc}(G)$=2 if and only if $H_i=\hat{K_4}$ for $i=1,2$.
\end{lem}
\begin{proof}
Suppose $\mathit{scfc}(G)$=2. Let $H_1=\hat{D_3}$ ($see \ Figure \ \ref{H}$). If $G$ is constructed by identifying the pendent edge of $H_1$ with $H_2\in\{\hat{K_4},\hat{D_3}\}$, then there is a forced 4-path $t_2t_1ys_1s_2$, a contradiction by Lemma \ref{path with forced}. Clearly, $\mathit{scfc}(G)$=2 when $H_1=H_2=\hat{K_4}$ under the edge-coloring in Figure \ref{N}.

\end{proof}
\begin{figure}[!htb]
\centering
\includegraphics[width=0.4\textwidth]{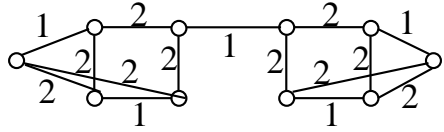}
\caption{The graph $N$ }\label{N}
\end{figure}

Combining Lemma \ref{scfcF0(k)=2}, Lemma \ref{singleedge}, Theorem \ref{scfc=spc=2} and Theorem \ref{oneforcededge}, we get the following result.
\begin{lem}\label{scfconeforced}
Let $G$ be a cubic graph with exactly one forced edge. Then $\mathit{scfc}(G)=2$ if and only if $G\cong F_0(k)$ for $k\in \{2,4\}$ or $G\cong N$.
\end{lem}

Before proceeding, we need one more definition.
\begin{df}
Let $G$ be a connected graph. The \emph{forced graph} of $G$ is obtained
from $G$ by replacing each forced edge $uv $ (if any) by two pendant edges $uu'$ and
$vv'$, where $u'$ and $v'$ are two new vertices with respect to the forced edge $uv$. Each
component of the forced graph of $G$ is called a \emph{forced branch} of $G$, and the new pendant
edge $uu'$ in the forced branch is called a \emph{forced link} of $G$. For each forced edge $uv$ of $G$,
we call $uu'$ and $vv'$ the \emph{twin links} corresponding to the forced edge $uv$. In the case that a forced link $uu'$
and its twin link $vv'$ are contained in a common forced branch of
$G$, we say that $uu'$ is a \emph{selfish link}.
\end{df}

\begin{thm}\label{atleasttwoforced}\upshape\cite{HY}
Let $G$ be a cubic graph containing at least two forced edges, and let $H_1,H_2,\cdots,H_r$ be the forced branches of $G$. Then $\mathit{spc}(G)=2$ if and only if $H_i\in \mathcal{H}$ for $i=1,2,\cdots,r$, and there are $2$-$SPC$ (strong proper connection number being $2$) patterns $p_1, p_2, \cdots, p_r$ of $H_1, H_2, \cdots, H_r$, respectively, such that each pair of twin links receive the same color.
\end{thm}

\begin{lem}\label{scfcatleasttwoforced}
Let $G$ be a cubic graph containing at least two forced edges, and let $H_1,H_2,\cdots,H_r \in \mathcal{H}$ be the forced branches of $G$. Then $\mathit{scfc}(G)$=2 if and only if $G\in\mathcal{L}$, demonstrated in Figure \ref{L}.
\end{lem}
\begin{figure}[!htb]
\centering
\includegraphics[width=0.6\textwidth]{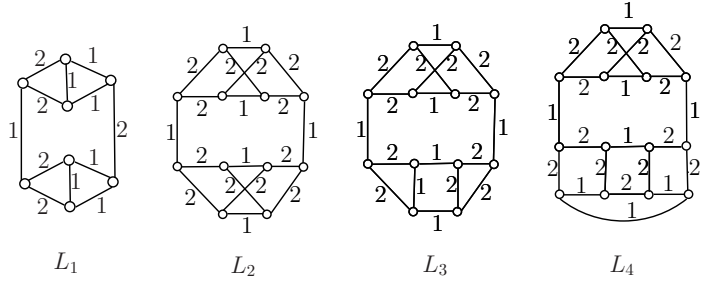}
\caption{The graph class : $\mathcal{L}$ }\label{L}
\end{figure}
\begin{proof}
Let $G^*$ be the \emph{forced graph} of $G$. According to the choices of $H$, we distinguish the following cases.

Suppose that $H_i\in\{\hat{K_4},\hat{D_3},\tilde{K_{3,3}},\tilde{Q_3}\}$. Clearly, there is no selfish link for $H_i$. Otherwise, there is no  forced edge in $G$. Suppose that $G^*$ is constructed by $H_1,H_2\in\{\hat{K_4},\hat{D_3}\}$. Then there is only one forced edge, a contradiction. Suppose that $G^*$ is constructed by $H_1\in\{\tilde{K_{3,3}},\tilde{Q_3}\}$ and $H_2,H_3\in\{\hat{K_4},\hat{D_3}\}$. Then there is a forced 5-path between the two forced edges, a contradiction. Suppose that $G^*$ is constructed by $H_1,H_2\cong \tilde{Q_3}$. Clearly, the two forced edges are contained in a forced cycle $C_8$ in $G$, which induces a forced 4-path, a contradiction. If $G^*$ is constructed by $H_1,H_2\cong \tilde{K_{3,3}}$ or $H_1\cong \tilde{K_{3,3}}, H_2\cong \tilde{Q_3}$, then $G\cong L_2$ or $L_4$ with $\mathit{scfc}(G)=2$ by the edge-coloring in Figure \ref{L}. Suppose that $H_i\in\{F_0^*(k),F_1(k)\}$. If $G^*\cong F_0^*(k)$, then there is at most one forced edge in $G$, a contradiction. If $G$ is constructed by identifying the pendent edges of $H_i\in\{F_0^*(k),F_1(k)\}$ to a single edge, then we can check that $G\cong L_1$. In the remaining case, $G\cong L_3$.
\end{proof}

Finally, Combining Theorem \ref{scfc=spc=2}, Lemma \ref{scfcatleasttwoforced} and Theorem \ref{atleasttwoforced}, we have our main theorem
of this section.
\begin{thm}
Let $G$ be a cubic graph. Then $\mathit{scfc}(G)=2$ if and only if
  \begin{center}
  $G\in\{L_1,L_2,L_3, L_4,N,C_{l}\Box K_2, M_{2r},F_0(k)\}$,
  \end{center}
  where $l\in\{3,4,6\}$, $3\leq r\leq7$ and $k\in\{2,4\}$.
\end{thm}

\end{document}